\documentclass[journal]{IEEEtran}
\usepackage{xspace}        % Handle spacing issues
\usepackage{mathtools}
\usepackage[hidelinks]{hyperref}
\usepackage{bbm}
\usepackage{url}
\usepackage{cleveref}

\usepackage{siunitx}
\usepackage{booktabs}       % professional-quality tables
\usepackage{multirow}
\usepackage{algorithmic}
\usepackage[ruled,vlined]{algorithm2e}       % algorithm
\usepackage{amsthm,amsmath,amssymb}
\usepackage{color}
\usepackage{comment}
\usepackage{tikz}
\usepackage{pgfplots}
\usepackage{bbm}
\usepackage{cite}
\usepackage{graphicx}
\usepackage{textcomp}

\newtheorem{lemma}{Lemma}
\newtheorem{theorem}{Theorem}
\newtheorem{assumption}{Assumption}

\ifCLASSINFOpdf
\else
\fi
\begin{document}
\title{First- and Zeroth-Order Learning \\ in  Asynchronous Games}

\author{Zifan Wang, Xinlei Yi, Michael M. Zavlanos, and Karl H. Johansson
        % <-this % stops a space
\thanks{* This work was supported in part by Swedish Research Council Distinguished Professor Grant 2017-01078, Knut and Alice Wallenberg Foundation, Wallenberg Scholar Grant, Swedish Strategic Research Foundation SUCCESS Grant FUS21-0026, and AFOSR under award \#FA9550-19-1-0169. (Corresponding author: ***.)}
\thanks{Zifan Wang and Karl H. Johansson are with the Division of Decision and Control Systems, School of Electrical Enginnering and Computer Science, KTH Royal Institute of Technology, and also with Digital Futures, SE-10044 Stockholm, Sweden. Email: \{zifanw,kallej\}@kth.se.}
\thanks{Xinlei Yi is with the Department of Control Science and Engineering, College of Electronics and Information Engineering, Tongji University. Email: xinleiyi@tongji.edu.cn.}
\thanks{Michael M. Zavlanos is with the Department of Mechanical Engineering and Materials Science, Duke University, Durham, NC, USA. Email: michael.zavlanos@duke.edu.}
}


\markboth{Journal of \LaTeX\ Class Files,~Vol.~14, No.~8, August~2015}%
{Shell \MakeLowercase{\textit{et al.}}: Asynchronous Learning in Online Convex Games}


\newcommand{\EE}{\mathbb{E}}
\newcommand{\SSS}{\mathbb{S}}
\newcommand{\BB}{\mathbb{B}}

\allowdisplaybreaks

\maketitle

\begin{abstract}
This paper investigates the discrete-time asynchronous games in which noncooperative agents seek to minimize their individual cost functions. Building on the assumption of partial asynchronism, i.e., each agent updates at least once within a fixed-length time interval, we explore the conditions to ensure convergence of such asynchronous games. The analysis begins with a simple quadratic game from which we derive tight convergence conditions through the lens of linear control theory. Then, we provide a quasidominance condition for general convex games. Our results demonstrate that this condition is stringent since when this condition is not satisfied, the asynchronous games may fail to converge. We propose both first- and zeroth-order learning algorithms for asynchronous games, depending on the type of available feedback, and analyze their last-iterate convergence rates. Numerical experiments are presented on economic market problems to verify our results.

% Specifically, we consider the partial asynchronism model and explore the class of games that ensure convergence for asynchronous updates.
% We start with a simple quadratic game and obtain the required condition for convergence based on linear control theory.
% Motivated by the simple case, we provide a quasidominance condition for general games that guarantee convergence for asynchronous updates in general games. 
% We further illustrate that the asynchronous learning in games may fail to converge when this quasidominance condition fails to hold, which indicates that the proposed condition is tight.
% We propose both first- and zeroth-order learning algorithms for asynchronous games, and analyze their convergence.
% Numerical experiments on economic market problems verify our analysis.

% asynchronous games, discrete time, agent update at different time steps;

% partial asynchronism;

% sufficient condition for convergence in asynchronous games;

% simple case, linear control theory, general case, quasidominance condition; stringent;

% Given condition, prove convergence for first-order and zeroth-order algs;

\end{abstract}


% \fi
\IEEEpeerreviewmaketitle

\section{Introduction}

Convex games find applications in many domains ranging from online marketing \cite{lin2021doubly} to transportation networks \cite{sessa2019no}.
In these games, agents aim to minimize their loss functions through interactions with other agents.
This interaction process involves each agent taking actions simultaneously and receiving feedback based on the collective actions of all agents. 
Using this feedback, each agent optimizes their actions.
The overall dynamics of the system are usually analyzed using the notion of the Nash equilibrium, which represents a stable point where no agent has an incentive to deviate.
There are numerous works that explore learning in different kinds of games \cite{bravo2018bandit,lin2020finite,drusvyatskiy2022improved,tatarenko2018learning,duvocelle2023multiagent,mazumdar2020gradient}.
For example, \cite{mazumdar2020gradient} analyzes the limiting behavior of continuous-time gradient-based dynamics for several classes of games using dynamical systems theory. 
Moreover, \cite{lin2020finite} proposes a discrete-time first-order gradient descent algorithm for $\lambda$-coercive games, demonstrating that it achieves the last-iterate convergence to a Nash equilibrium.
The works \cite{bravo2018bandit,drusvyatskiy2022improved} investigate scenarios where agents only access zeroth-order oracles, i.e., function evaluations, and establish convergence for strongly monotone games.

In this paper, we consider the more practical case of asynchronous games, in which agents update their actions at their own pace. 
Specifically, we focus on the discrete-time asynchronous games, where each agent adjusts its actions according to individually determined schedules.
This scenario is common in practical applications.
For example, in Cournot games, multiple companies set production levels to minimize their own costs, but they probably do so on different timelines. Some companies may adjust their actions monthly, while others make adjustments daily, seasonally, or even randomly. 
Despite its practical relevance, the analysis of such asynchronous games is lacking in existing literature.
Motivated by the above discussions, our research aims to systematically investigate games that ensure convergence in asynchronous settings. Moreover, we seek to develop algorithms and analyze their convergence rates to better understand the underlying processes.

Our contributions are detailed as follows. We provide a detailed discussion on the class of games that ensure convergence of asynchronous games. 
Focusing on the partial asynchronism mechanism, which ensures that each agent updates at least once within a fixed-length time interval, we establish the condition for the convergence of dynamics of asynchronous games. 
Specifically, we begin with a simple quadratic game from which we analyze the whole system through the existing linear control theory. 
Inspired by the simple case, we introduce a quasidominance condition for general convex games, which we show is stringent since the failure of satisfaction may lead to divergence.
It is worth mentioning that the strong monotonicity condition, which can be implied by the quasidominance condition, is not sufficient to ensure the convergence of asynchronous games.
Then, we propose first- and zeroth-order asynchronous learning algorithms tailored for different feedback available to agents. The first-order algorithm is designed for scenarios where agents can observe others' actions and compute the first-order gradient, while the zeroth-order one is suited for situations where agents only receive cost evaluation feedback. We provide theoretical convergence guarantees for both algorithms.
Finally, we validate our results by conducting numerical experiments on a Cournot game.

Related to our work is the literature on asynchronous distributed optimization \cite{assran2020advances,zhang2019asyspa,feyzmahdavian2023asynchronous,wu2023asynchronous,wu2023delay}. 
In these works, multiple agents cooperatively minimize a cost function through communication, and asynchronous algorithms are designed to accelerate the computation in large-scale systems.
However, distributed optimization is framed as a multi-agent cooperative problem, which contrasts sharply with the competitive nature of games.
Moreover, in multi-agent games, each agent’s cost function depends on the actions of others, making the dynamics of asynchronous updates more involved.
Consequently, the methodologies developed in the field of asynchronous distributed optimization are not applicable to our context.
Perhaps most relevant to our research are the studies on asynchronous learning in networked games \cite{cenedese2020asynchronous,cenedese2021asynchronous,yi2019asynchronous,zhu2022asynchronous}. 
For example, \cite{cenedese2021asynchronous} considers asynchronous proximal dynamics in convex games and establishes convergence when the degree of asynchrony is limited. 
However, asynchronous algorithms in \cite{cenedese2020asynchronous} and \cite{cenedese2021asynchronous} allow only one agent to update at each iteration, whereas our algorithms allow multiple agents to update simultaneously. 
Moreover, the techniques used for analysis are quite different. While these prior works analyze the asymptotic convergence, we provide a new non-asymptotic convergence analysis.
Besides, these works assume access to first-order gradients for each agent, which may not be feasible when agents are unwilling to share their actions.
To solve this issue, we additionally consider a setting where agents rely solely on function evaluations and employ zeroth-order gradient estimates for action updates.

Another related research line is learning in synchronous games \cite{mertikopoulos2019learning,belgioioso2020distributed,bravo2018bandit,drusvyatskiy2022improved,tatarenko2018learning,lin2020finite}. 
For example, \cite{lin2020finite} proposes a first-order gradient descent algorithm for $\lambda$-coercive games and demonstrates that it achieves last-iterate convergence to a Nash equilibrium.
In the case where agents only access function evaluations, \cite{bravo2018bandit} proposes a zeroth-order optimization method and shows its convergence to the Nash equilibrium for strongly monotone games.
However, common in these works is that the agents perform synchronous updates, and their methods cannot be directly extended to the asynchronous setting considered here.
It is worth mentioning that asynchronous games are related to games with delayed rewards \cite{mertikopoulos2020gradient} since asynchronous updates can result in outdated information for each agent. 
However, the delayed setting in \cite{mertikopoulos2020gradient} is fundamentally different from asynchronous setting: asynchronous setting emphasizes the independence and lack of coordination between agents, whereas the delayed setting focuses on the lag in the feedback received by each agent.

The rest of the paper is organized as follows. In Section~\ref{sec:1_problem}, we formally define the asynchronous games. In Section~\ref{sec:2_condition}, we discuss the condition required for convergence of asynchronous games.
We provide first- and zeroth-order asynchronous learning algorithms in Sections~\ref{sec:3_FO} and \ref{sec:4_ZO}, respectively.
In Section~\ref{sec:5_experiments}, we numerically verify our methods using a Cournot game example. Finally, we conclude the paper in Section~\ref{sec:6_conclusion}.

% Learning in games;

% asynchronous;

% In this paper,

% Related works.

% %%%%%%%%%%
\section{Problem Definition}
\label{sec:1_problem}

Consider a repeated game involving $N$ non-cooperative agents, whose goals are to minimize their own cost functions through interactions with others. For each agent $i$, the cost function is defined as $C_i(x_i,x_{-i})$, where $x_i$ is the action of agent $i$, and $x_{-i}$ denotes the actions of all agents except agent $i$. Suppose that agent $i$'s action is constrained in a closed convex set $\mathcal{X}_i \in \mathbb{R}^d$ and $\mathcal{X}=\Pi_{i=1}^N\mathcal{X}_i$ is the joint action space. We assume that $\mathcal{X}_i$ contains the ball with radius $R$ centered at the origin and has a bounded diameter $D>0$, for all $i=1,\ldots,N$.

Each agent $i$ aims to find the best actions that minimize the cost $C_i$ in an iterative interaction process by virtue of the received feedback. This feedback may consist of either first-order gradient information or zeroth-order function values, depending on the mechanisms of the game.
This paper explores the dynamics of such interactions over a discrete-time horizon $t=1,\ldots,T$. 
Importantly, we consider a practical scenario in which agents do not update their actions synchronously but instead follow individually determined schedules.
To model this, we define $\mathbb{T}^i$ as the set of time steps at which agent $i$ plans to update their actions at the subsequent time step.
For example, if the current time is $\tau$ and $\tau \in \mathbb{T}^i$, then agent $i$ will update their actions at the next time step $\tau+1$.

If the agent updates their actions in a completely asynchronous manner, the system's dynamics are expected to deteriorate. Therefore, we make the following assumption, which is denoted as partial asynchronism in \cite{bertsekas2015parallel}.
\begin{assumption}\label{assump:delay}
There exists $B>0$ such that each agent updates at least once in the time interval $[t,t+B)$, for all $t$.
\end{assumption}
Assumption~\ref{assump:delay} ensures that within any time interval of length $B$, each agent must perform at least one update. Here, we do not have any other restriction on the mechanism of asynchronism.

In the context of the asynchronous games described above, our goal is to investigate the class of games where gradient-based asynchronous algorithms still converge to Nash equilibria.
Besides, we aim to develop specific algorithms and analyze their convergence rates for the asynchronous setting.

% %%%%%%%%%%
\section{Stringent Condition for Convergence of Asynchronous Games}\label{sec:2_condition}
In this section, we explore the class of games in which the dynamics of asynchronous games are guaranteed to converge. We begin with a simple quadratic game such that its gradient dynamics can be analyzed through existing linear control theory. 
Specifically, in this simple case, we aim to explore (Q1) the condition required for convergence of gradient descent dynamics in synchronous games; (Q2) whether convergence in synchronous games guarantees convergence in asynchronous games; and (Q3) if not, the condition required for convergence in asynchronous games.
Then, we extend it to more general games.

\subsection{Simple Case Illustration}
Consider a game with each agent $i$ having a quadratic cost $C_i(x)$ and gradient $\nabla_i C_i(x) = \sum_{j}\nabla_{ij} C_i(x) x_j$, where $\nabla_{ij} C_i(x)$ denotes the partial derivative of $C_i$ with respect to $x_i$ and $x_j$, and is assumed to be constant. For ease of illustration, we assume $x_i \in \mathbb{R}$ in this simple case. 

In synchronous games, each agent $i$ updates their actions according to the gradient descent dynamics $x_{i,t+1} = x_{i,t} - \eta \nabla_i C_i(x_t)$, where $x_t=(x_{1,t}^\top,\dots,x_{N,t}^\top)^\top$. The dynamics of the entire system can be expressed as 
\begin{align}\label{eq:dynamics:syn}
    x_{t+1} = x_{t} - \eta J x_t,
\end{align}
where the Jacobian matrix $J$ is constant and defined as
\begin{align*}
    J := \begin{bmatrix}
        \nabla_{11} C_1(x) & \cdots & \nabla_{1N} C_1(x) \\
        \vdots & \ddots & \vdots \\
        \nabla_{N1} C_N(x) & \cdots & \nabla_{NN} C_n(x)
    \end{bmatrix}
    = \begin{bmatrix} J^1 \\ \vdots \\ J^N\end{bmatrix}.
\end{align*}
Here, $J^i$ denotes the $i$-th row of the matrix $J$.
It is easy to verify that the Nash equilibrium is the original point, which coincides with the stable point of the linear dynamics \eqref{eq:dynamics:syn}.
Therefore, the convergence of synchronous games is equivalent to the stabilization of the linear dynamics.
It is well known from linear control theory that the dynamics \eqref{eq:dynamics:syn} converges if and only if $-J$ is Hurwitz, i.e., every eigenvalue of $-J$ has a strictly negative real part. This answers the question (Q1). 
Besides, from \cite{FB-CTDS}, $-J$ is Hurwitz if and only if there exists $P>0$ such that $ -J^{\top}P - PJ<0$. 
It can be verified that this condition is equivalent to the diagonally strictly concave condition in \cite{rosen1965existence}, which is termed as monotonicity condition in some later works \cite{bravo2018bandit}.
This observation aligns with the common sense that the gradient descent dynamics of monotone games converge to the equilibrium.

In asynchronous games, the gradient descent dynamics can be expressed as 
\begin{align}\label{eq:dynamics:asyn}
    y_{t+1} = y_t - \eta J_t y_t,
\end{align}
where $J_t^i$ equals $J^i$ if $t\in \mathbb{T}^i$ and $\textbf{0}_{1\times N}$ otherwise. 
Iteratively using \eqref{eq:dynamics:asyn} yields
\begin{align}\label{eq:dynamics:asyn2}
    y_{t+B} &= \Pi_{k=t}^{t+B-1}(I - \eta J_k)y_t \nonumber \\
    &=y_t - \big(\eta (\sum_{k=t}^{t+B-1}J_k) + \mathcal{O}(\eta)\big) y_t \nonumber \\
    & = y_t - \big(\eta H_t + \mathcal{O}(\eta)\big) y_t,
\end{align}
where $H_t =\sum_{k=t}^{t+B-1}J_k= A_t J $, and $\mathcal{O}(\eta)$ denotes the higher-order terms of $\eta$ that can be ignored when $\eta$ is small enough. The matrix $A_t$ is a diagonal matrix whose $i$-th element, $a_{i,t}$, represents the number of updates performed by agent $i$ during the time interval $[t,t+B)$. Given Assumption~\ref{assump:delay}, we have $a_{i,t}\geq 1 $ and thus $A_t$ is always diagonally positive definite.
From linear control theory again, the dynamics \eqref{eq:dynamics:asyn2} converges if $-A_t J$ is Hurwitz.
Obviously, the Hurwitz property of $-J$ cannot guarantee the Hurwitz property of $-A_t J$ for any $A_t$. 
For example, when 
\begin{align*}
    J = \begin{bmatrix}
    0.1 & -2 & 1 \\
    -2 & 0.2 & 4 \\
    -3 & -4 & 1.7
    \end{bmatrix}, \;
    A_t = \begin{bmatrix}
    2 & 0 & 0 \\
    0 & 1 & 0 \\
    0 & 0 & 1
    \end{bmatrix},
\end{align*}
it can be verified that $-J$ is Hurwitz but $-A_t J $ is not.
Therefore, the answer to the question (Q2) is negative.

Based on the above analysis, the question (Q3) in the simple quadratic game can be detailed as follows: What is the condition required for the matrix $J$ to ensure that $-A_t J$ is Hurwitz for the diagonal matrix $A_t$ with diagonal elements being positive integers?
The condition is that $J$ is quasidominant \cite{FB-CTDS}, i.e., there exists a vector $r\in \mathbb{R}^N$ with $r_i>0$ such that $r_i J_{ii} > \sum_{j \neq i} r_j |J_{ij}|$ for all $i=1,\ldots,N$.
To see this, we first have that $A_t J$ is quasi-dominant since $r_i a_{i,t} J_{ii} > \sum_{j \neq i} r_j a_{i,t} |J_{ij}|$ for any positive $a_{i,t}$.
Then, from Remark 2.6 in \cite{FB-CTDS}, it holds that $-A_t J$ is Hurwitz. Therefore, the above quasidominant condition guarantees the convergence of asynchronous dynamics in the simple quadratic game.

\subsection{General Function Case}
Motivated by the quasidominant condition observed in the above simple scenario, this section presents the condition required for the convergence of asynchronous games in general cases.
The specific condition is detailed in the following assumption.
\begin{assumption}\label{assump:cost} (Quasidominance)
The cost function $C_i(x_i,x_{-i})$ is $\mu_i$-strongly convex in $x_i$ for every $x_{-i}$ and $\nabla_i C_i(x)$ is $L_{ij}$-Lipschitz in $x_j$, for $i=1,\ldots,N$. Besides, there exists a positive vector $r=(r_1,\ldots,r_N)$ such that $r_i \mu_i > \sum_{j \neq i} r_j L_{ij}$, for $i=1,\ldots,N$.
\end{assumption}

Assumption~\ref{assump:cost} states that the influence of other agents on each agent $i$, represented by the parameters $L_{ij}$, is sufficiently small.
With Assumption~\ref{assump:cost}, we can guarantee the existence and uniqueness of the Nash equilibrium. This is established by showing that Assumption~\ref{assump:cost} ensures the game is strictly monotone, as presented in the following lemma.

\begin{lemma}
Given Assumption~\ref{assump:cost}, the game with the cost functions $C_i$ is strictly monotone, i.e., there exist positive constants $\lambda_i$ such that
\begin{align*}
    \sum_i \lambda_i \langle \nabla_i C_i(x) - \nabla_i C_i(y) ,x_i - y_i \rangle > 0,
\end{align*}
for all $x,y \in \mathcal{X}$, $x \neq y$.
\end{lemma}
\begin{proof}
Define the matrix 
\begin{align*}
    Q = \begin{bmatrix}
        \mu_1 & \cdots & -L_{1N} \\
        \vdots & \ddots & \vdots  \\
        -L_{N1} & \cdots & \mu_N
    \end{bmatrix}.
\end{align*}
Given Assumption~\ref{assump:cost}, we have that the matrix $Q$ is quasidominant, which implies that $-Q$ is $M$-Hurwitz \cite{FB-CTDS}. By the Hurwitz Metzler Theorem \cite{FB-LNS}, there exists a diagonal matrix $\Lambda>0$ such that $-\Lambda Q - Q^{\top} \Lambda <0$.
Setting $\lambda_i$ as the $i$-th element of the matrix $\Lambda$, we have
\begin{align*}
    &\sum_i \lambda_i \langle \nabla_i C_i(x) - \nabla_i C_i(y) ,x_i - y_i \rangle \nonumber \\
    &=\sum_i \lambda_i \langle \nabla_i C_i(x) - \nabla_i C_i(y_i,x_{-i}) ,x_i - y_i \rangle \nonumber \\
    &\quad + \sum_i \lambda_i \langle \nabla_i C_i(y_i,x_{-i}) - \nabla_i C_i(y) ,x_i - y_i \rangle \nonumber \\
    &\geq \sum_i \lambda_i \mu_i \left\| x_i - y_i\right\|^2 - \sum_i \lambda_i \sum_{j \neq i} L_{ij}\left\| x_j - y_j \right\| \left\| x_i - y_i\right\| \nonumber \\
    & = z^\top \Lambda Q z = z^\top \frac{\Lambda Q + Q^{\top} \Lambda}{2} z,
\end{align*}
where we define $z = [\left\|x_1 - y_1 \right\|,\cdots, \left\|x_N - y_N \right\|]^{\top}$. The inequality follows from the strong convexity of $C_i$ and Lipschitz continuous property of $\nabla_i C_i$. Since $\Lambda Q + Q^{\top} \Lambda >0$, we obtain the desired result. 
\end{proof}

As shown by \cite{rosen1965existence}, monotone games admit a unique Nash equilibrium. Therefore, games that fulfill Assumption~\ref{assump:cost} also exhibit a unique Nash equilibrium. Throughout this paper, we denote by $x^*$ the unique Nash equilibrium.

The condition $r_i \mu_i > \sum_{j \neq i} r_j L_{ij}$ generalizes the one observed in the simple case. 
Although this condition may appear conservative at first glance, it is indeed stringent since failure of satisfaction may not lead to convergence in even synchronous dynamics. Consider a simple example where we have $J=[1,-1;-1,1]$. In this case, the quasidominance condition fails to hold. It can be easily verified that the dynamics $x_{t+1}=x_t - A J x_t$ fail to converge for any diagonal matrix $A$ with diagonal elements being positive integers.

% %%%%%%%%%%
\section{First-order Asynchronous Games}\label{sec:3_FO}
In this section, we propose a first-order gradient descent algorithm for asynchronous games and analyze its convergence.

We assume that during the learning process, each agent can observe the actions of other agents. When they decide to update their actions, they utilize the first-order gradient.
The detailed algorithm is presented in Algorithm~\ref{alg:FO}.
Specifically, at each time step $t$, if $t\in \mathbb{T}^i$,
agent $i$ collects first-order gradient information and updates its action for the next time step. Otherwise, if $t \not\in \mathbb{T}^i$, agent $i$ retains the same action.
The update strategy of the agent $i$ can be expressed as
\begin{align}\label{eq:update:FO}
    x_{i,t+1} = \left\{ \begin{array}{cc} \mathcal{P}_{\mathcal{X}_i}(x_{i,t}-\eta \nabla_i C_i(x_t) ), & t\in T^i, \\ x_{i,t}, & {\rm{otherwise}} ,\end{array} \right.
\end{align}
where $\mathcal{P}_{\mathcal{X}_i}$ denotes the projection onto the set $\mathcal{X}_i$.
\begin{algorithm}[t]
\caption{First-order asynchronous games} \label{alg:FO}
\begin{algorithmic}[1]
    \STATE \textbf{Input}:Step size $\eta$, initial value $x_{i,1}$, $i=1,\ldots,N$, and length of interval $T$.
    \FOR{$ t=1,\ldots,T$}
        \FOR{agent $i$}
        \STATE Play the action $x_{i,t}$
        \ENDFOR
        \FOR{agent $i$}
            \IF{ $t\in T^i$}
                \STATE $x_{i,t+1} = \mathcal{P}_{\mathcal{X}_i}(x_{i,t}-\eta \nabla_i C_i(x_t) )$
            \ELSE
                \STATE $x_{i,t+1}=x_{i,t}$
            \ENDIF
        \ENDFOR
    \ENDFOR
\end{algorithmic}
\end{algorithm}

% \subsection{Convergence Analysis}

To facilitate the analysis, we denote by $M_{i,t}$ the total update times of agent $i$ in the time interval $[t,t+B)$. Due to partial asynchronism as assumed in Assumption~\ref{assump:delay}, we have $M_{i,t}\geq 1$, for all $i$ and $t$.
For the $m$-th update, $m=1,\ldots,M_{i,t}$, we denote the corresponding update time step by $\tau_{i,t}^m$ with $\tau_{i,t}^m \in T^i$.
Specifically, $\tau_{i,t}^m$ denotes the time step of the $m$-th update of agent $i$ in the interval $[t,t+B)$. 

Based on the definitions above, in the time interval $[t,t+B)$, the update equation \eqref{eq:update:FO} can be equivalently written as 
\begin{align}\label{eq:FO:t1}
    x_{i,\tau_{i,t}^m+1} = \mathcal{P}_{\mathcal{X}_i}(x_{i,\tau_{i,t}^m}-\eta \nabla_i C_i(x_{\tau_{i,t}^m}) ),
\end{align}
for $m=1,\ldots, M_{i,t}$. By definition of $\tau_{i,t}^m$, we have $x_{i,\tau_{i,t}^1} = x_{i,\tau_{i,t}^1 -1} =\cdots= x_{i,t}$, $x_{i,\tau_{i,t}^{M_{i,t}}+1} = x_{i,t+B}$ and $ x_{i,\tau_{i,t}^{m}} = x_{i,\tau_{i,t}^{m-1}+1}$ for $m\geq 2$. By defining $\tau_{i,t}^{0} := t-1$, we have the following consistent equality that 
\begin{align}\label{eq:FO:t2}
    x_{i,\tau_{i,t}^{m}} = x_{i,\tau_{i,t}^{m-1}+1}, \quad \forall m=1,\ldots,M_{i,t}.
\end{align}

We make the following assumption on the gradient of the cost function before the analysis of convergence.
\begin{assumption}\label{assump:grad:bound}
For each agent $i=1,\ldots,N$, $\left\|\nabla_i C_i(x) \right\|\leq U$ for all $x\in \mathcal{X}$.
\end{assumption}
Assumption~\ref{assump:grad:bound} is common in the literature, see, e.g., \cite{duvocelle2023multiagent}. Now we are ready to present the convergence result for Algorithm~\ref{alg:FO}. The proof can be found in the Appendix.

\begin{theorem}\label{thm:FO}
Let Assumptions~\ref{assump:delay}--~\ref{assump:grad:bound} hold and select $\eta = \frac{ B\ln(T/B)}{ \varepsilon T}$, where $\varepsilon := \min_i \left\{ \mu_i - \frac{1}{r_i} \sum_{j\neq i}r_j L_{ij} \right\}>0$. Then, Algorithm~\ref{alg:FO} satisfies that 
% \begin{align}
%     \frac{1}{T}\sum_{t=1}^T \max_i\left\| x_{i,t} - x_i^*\right\|^2 = \mathcal{O}(\frac{B^{3/2}}{T^{1/2}}).
% \end{align}
\begin{align}
    \max_i\left\| x_{i,T} - x_i^*\right\|^2 = \mathcal{O}(\frac{ B^3 \ln (T/B)}{T}).
\end{align}
\end{theorem}
Theorem~\ref{thm:FO} establishes last-iterate convergence for Algorithm~\ref{alg:FO}, showing that a larger value of $B$, which indicates more complicated asynchronism, leads to slower convergence. This convergence guarantee implies $ \left\| x_T - x^*\right\|^2 = \mathcal{O}(\frac{ B^3 \ln (T/B)}{T})$ since $\left\| x_T - x^*\right\|^2 = \sum_{i} \left\| x_{i,T} - x_i^*\right\|^2 \leq N \max_i\left\| x_{i,T} - x_i^*\right\|^2$. Notably, the convergence rate with respect to $T$ matches the best-known result for online learning in synchronous games, as presented in \cite{jordan2024adaptive}.

The convergence is demonstrated using a measure of the maximum distance of each agent's error to the Nash equilibrium. This measure comes from the proof where we use the Lyapunov function $V_t = \max_i \frac{\left\| x_{i,t} - x_i^* \right\|^2}{r_i^2}$.  Dividing by $r_i^2$ in the Lyapunov function serves to normalize each agent's error, balancing their contributions based on their own properties such as strong convexity and coupling effects. This scaling allows for a unified convergence analysis in asynchronous games.

\section{Zeroth-Order Asynchronous Games}\label{sec:4_ZO}
In this section, we propose a zeroth-order gradient descent algorithm for asynchronous games and analyze its convergence.

\begin{algorithm}[t]
\caption{Zeroth-order asynchronous games} \label{alg:ZO}
\begin{algorithmic}[1]
    \STATE \textbf{Input}: Step size $\eta$, perturbation size $\delta$, initial value $x_{i,1}$, $\hat{x}_{i,0}$, $i=1\ldots,N$, and length of interval $T$.
    \FOR{$ t=1,\ldots,T$}
        \FOR{agent $i$}
            \IF{ $t\in T^i$}
                \STATE Sample $u_{i,t} \in \mathbb{S}^d$
                \STATE Play the action $\hat{x}_{i,t}=x_{i,t}+\delta u_{i,t} $
            \ELSE
                \STATE Play the action $\hat{x}_{i,t} = \hat{x}_{i,t-1}$
            \ENDIF
        \ENDFOR
        \FOR{agent $i$}
            \IF{ $t\in T^i$}
                \STATE $x_{i,t+1} = \mathcal{P}_{\mathcal{X}_i^{\delta}}(x_{i,t}-\eta g_{i,t} )$, where $g_{i,t}=\frac{d}{\delta} C_i(\hat{x}_t) u_{i,t}$
            \ELSE
                \STATE $x_{i,t+1}=x_{i,t}$
            \ENDIF
        \ENDFOR
    \ENDFOR
\end{algorithmic}
\end{algorithm}

We consider the case that each agent cannnot observe other agents' actions and only receives the feedback on cost evaluations. The specific algorithm is presented in Algorithm~\ref{alg:ZO}. At each time step $t$, if $t\in \mathbb{T}^i$, (i.e., agent $i$ is scheduled to update the action at the next time step $t+1$), agent $i$ becomes active to gather information for the action update. To do so, agent $i$ samples a random vector $u_{i,t}$ from the unit sphere $\mathbb{S}^d \in \mathbb{R}^d$ and plays the perturbed action  $\hat{x}_{i,t} = x_{i,t}+ \delta u_{i,t}$ at the time step $t$, where $\delta$ is the perturbation size.
If $t \not\in \mathbb{T}^i$, (i.e., agent $i$ is scheduled to keep the action), agent $i$ simply plays the perturbed action same as the previous step, i.e., $\hat{x}_{i,t}=\hat{x}_{i,t-1}$.
After agents play the perturbed actions, they update their actions as follows:
\begin{align}\label{eq:update:ZO}
    x_{i,t+1} = \left\{ \begin{array}{cc} \mathcal{P}_{\mathcal{X}_i^{\delta}}(x_{i,t}-\eta g_{i,t} ), & t\in T^i, \\ x_{i,t}, & {\rm{otherwise}},\end{array} \right. 
\end{align}
where $g_{i,t}=\frac{d}{\delta} C_i(\hat{x}_t) u_{i,t}$ and $\hat{x}_t=(\hat{x}_{1,t}^\top,\dots,\hat{x}_{N,t}^\top)^\top$. 
The projection set is defined as
$ \mathcal{X}_{i}^{\delta} = \{ x_{i}\in \mathcal{X}_{i} \vert \frac{1}{1-\delta/R}x_{i} \in \mathcal{X}_{i}\}$. The projection step guarantees the feasibility of the sampled action $\hat{x}_{i,t}$, since 
\begin{align*}
    \bigg(1-\frac{\delta}{R}\bigg) \mathcal{X}_{i} \oplus \delta \mathbb{S} &= \bigg(1-\frac{\delta}{R}\bigg) \mathcal{X}_{i} \oplus \frac{\delta}{R} R   \mathbb{S}\\
    & \subseteq \bigg(1-\frac{\delta}{R}\bigg)\mathcal{X}_{i} \oplus \frac{\delta}{R} \mathcal{X}_{i} = \mathcal{X}_{i}.
\end{align*}
Here, $\oplus$ denotes the Minkowski sum of two sets.

Same as the first-order case, we denote by $M_{i,t}$ the total update times of agent $i$ in the time interval $[t,t+B)$ and $\tau_{i,t}^m$ the time step for the $m$-th update, $m=1,\ldots,M_{i,t}$.
We obtain that the update formula \eqref{eq:update:ZO} is equivalent to 
\begin{align}\label{eq:ZO:t1}
    x_{i,\tau_{i,t}^m+1} = \mathcal{P}_{\mathcal{X}_i^{\delta}}(x_{i,\tau_{i,t}^m}-\eta g_{i,\tau_{i,t}^m} ).
\end{align}

Before the analysis of convergence, we present the following assumption on the cost function $C_i$.
\begin{assumption}\label{assump:C:bound}
For each agent $i=1,\ldots,N$, $|C_i(x)|\leq U_c$ for all $x\in \mathcal{X}$.
\end{assumption}
Assumption~\ref{assump:C:bound} is commonly employed in zeroth-order optimization for games, as discussed in works such as \cite{bravo2018bandit,lin2021doubly}.
Now we are ready to present the convergence analysis for Algorithm~\ref{alg:ZO}. The proof can be found in the Appendix.
\begin{theorem}\label{thm:ZO}
Let Assumptions~\ref{assump:delay}--\ref{assump:C:bound} hold, and select $\delta = \frac{B}{T^{1/3}}$ and $\eta = \frac{ B\ln(T/B)}{ \varepsilon T}$ where $\varepsilon = \min_i \mu_i - \frac{1}{r_i} \sum_{j\neq i}r_j L_{ij}$. Then, Algorithm~\ref{alg:FO} achieves convergence 
\begin{align}
    \max_i \EE \left\| x_{i,T} - x_i^*\right\|^2 = \mathcal{O}(\frac{B^2 \ln (T/B)}{T^{1/3}}).
\end{align}
\end{theorem}

Theorem~\ref{thm:ZO} establishes the convergence of zeroth-order asynchronous learning for the class of games that satisfy Assumption~\ref{assump:cost}.  A larger value of $B$, indicating more complex asynchronism, leads to slower convergence. The convergence rate with respect to $T$, specifically $\tilde{\mathcal{O}}(T^{-1/3})$, matches the result in \cite{bravo2018bandit} for strongly monotone games. Recently, \cite{lin2021doubly} improved the convergence result, achieving the optimal rate of $\tilde{\mathcal{O}}(T^{-1/2})$ using a mirror descent algorithm.  Investigating the optimal rate in the asynchronous setting remains a challenging problem, which we leave for future work.

Analyzing zeroth-order asynchronous games is notably more complex than dealing with the first-order case. In zeroth-order optimization, agents lack direct access to the gradients of their objective functions. Instead, they must rely on additional procedures to estimate it, i.e., by sampling random perturbations, playing perturbed actions, and observing the resulting changes in their cost evaluations. Each agent performs this estimation only when it intends to update its action in the next time step. However, due to asynchrony, agents do not necessarily perform these procedures simultaneously; their played actions may remain unchanged from some earlier time steps, and the corresponding random vectors are those sampled at those times. This lack of synchronization introduces disorder into the system's randomness.  Consequently, each agent's gradient estimate is no longer an unbiased estimate of a specific smoothed function, as it would be in synchronous games. Despite this challenge, we demonstrate that the quality of the gradient estimates in asynchronous games can still be maintained. Building on this fact, we construct our proof using the Lyapunov function $V_t^{\delta}=\max_i\left\| x_{i,t} - x_{\delta_i}^*\right\|^2$, where $x_{\delta_i}^* = (1-\frac{\delta}{R})x_{i}^*$.

Theorems~\ref{thm:FO} and \ref{alg:ZO} show that first- and zeroth-order learning achieves convergence rates $\mathcal{O}(\frac{ B^3 \ln (T/B)}{T})$ and $\mathcal{O}(\frac{B^2 \ln (T/B)}{T^{1/3}})$, respectively. 
While the first-order method has a worse dependence on the parameter $B$ compared to the zeroth-order method, it benefits from a faster improvement with respect to $T$.
Consequently, when $T^{2/3}>B$, the first-order method achieves a better convergence rate.

% %%%%%%%%%%
\section{Numerical Experiments}\label{sec:5_experiments}

In this section, we verify our theoretical results using a Cournot game example. 
Consider a market consisting of $N$ agents, where each agent $i$ contributes a supply quantity $x_i$ to the market and incurs the production cost $c_i x_i$. The aggregate supply from all agents determines the market price $p(x)=e_i - 0.5 J_{ii} x_i - \sum_{j \neq i} J_{ij} x_j$. Then, each agent $i$ incurs an overall cost $C_i(x) = -x_i p(x)+c_i x_i $. The objective for each agent is to minimize their individual costs through asynchronous gradient-based dynamics.

We first consider a setting that does not satisfy the quasidominant condition. Specifically, we set $J= [0.1, -2, 1; -2, 0.2, 4;-3, -4, 1.7]$, $e=[2.6, 2.1, 2.3]$, and $c=[0.2, 0.1, 0.5]$. It can be verified that $J$ is Hurwitz, so the synchronous dynamics converge, as shown in Fig.~\ref{fig_s1}.
For asynchronous updates, we consider that agents update periodically with periods $1,2,2$, respectively.
Following the discussion in Section~\ref{sec:2_condition}.A, agents will update twice, once, once, respectively, at every two time steps. We define $A = {\rm{diag}}([2,1,1])$ and find that $AJ$ is not Hurwitz. As shown in Fig.~\ref{fig_s1}, the asynchronous games diverge no matter how small the step size is.

Next, we consider a different configuration where we set 
$J = [1, -0.3, 0.4;0.2, 1, -0.5;0.5, 1.2, 2]$, $e=[1.6, 4.4, 1.0]$ and $c=[0.2, 0.1, 0.5]$. It can be easily verified that the above setting satisfies the quasidominant condition.
For asynchronous updates, we consider that agents update periodically with periods $[7,5,3]$, respectively.
As shown in Fig.~\ref{fig_s2}, both synchronous and asynchronous updates lead to convergence.
We also evaluate the convergence rate of first-order and zeroth-order asynchronous learning algorithms. Fig.~\ref{fig_FOZO} shows that the first-order algorithm converges much faster than the zeroth-order method.
Moreover, we examine the effect of large parameter values for $B$ by comparing two setups with update periods [7,5,3] and [17,13,7]. As shown in Fig.~\ref{fig_diffD}, large periods lead to slower convergence.

\begin{figure}[t] 
\begin{center}
\centerline{\includegraphics[width=0.9\columnwidth]{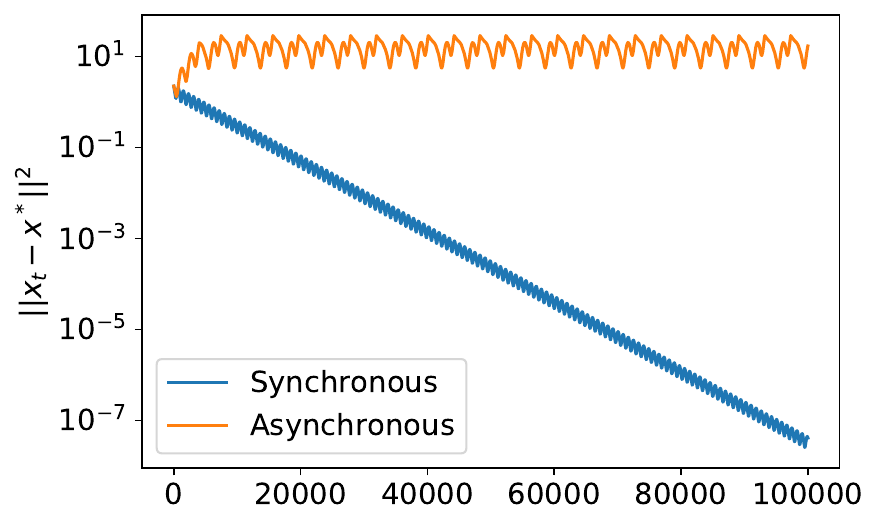}}
\caption{Convergence results of synchronous and asynchronous first-order dynamics when the quasidominant condition is not satisfied.}
\label{fig_s1}
\end{center}
\vskip -0.2in
\end{figure}

\begin{figure}[t] 
\begin{center}
\centerline{\includegraphics[width=0.9\columnwidth]{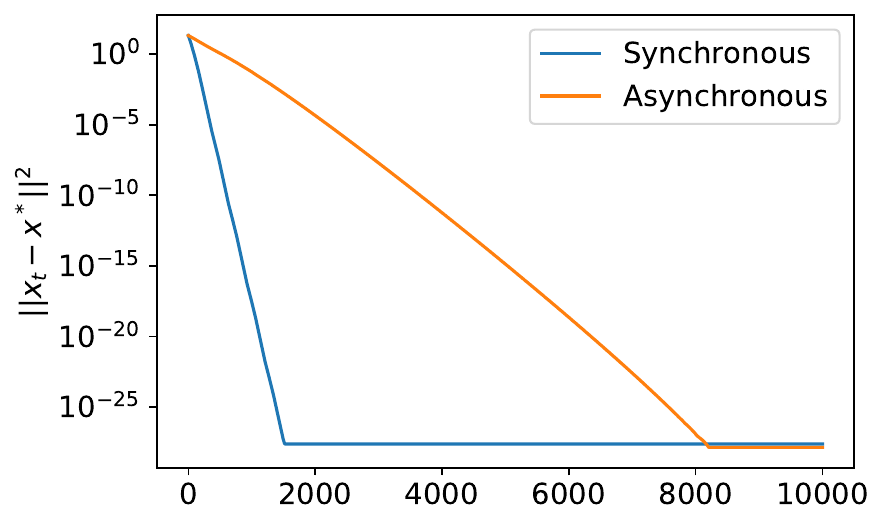}}
\caption{Convergence results of synchronous and asynchronous first-order dynamics when the quasidominant condition is satisfied.}
\label{fig_s2}
\end{center}
\vskip -0.2in
\end{figure}

\begin{figure}[t] 
\begin{center}
\centerline{\includegraphics[width=0.9\columnwidth]{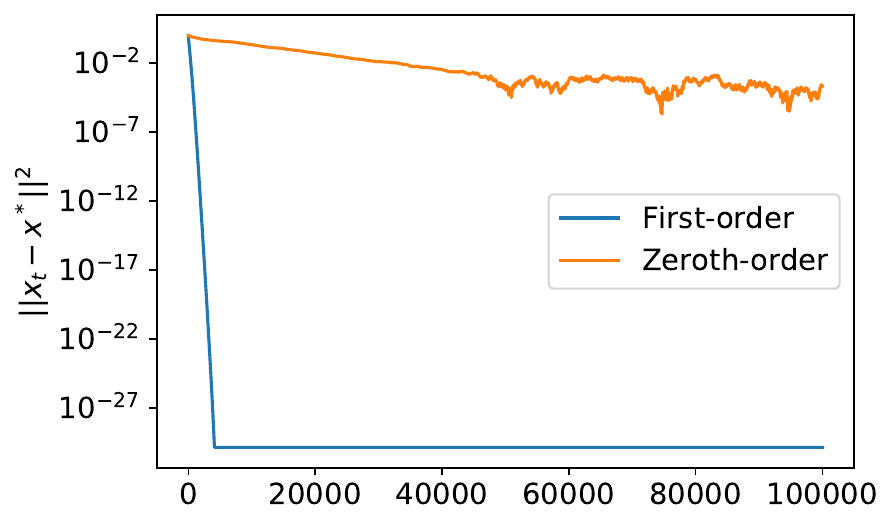}}
\caption{Comparitive results of Algorithms 1 and 2 in asynchronous games.}
\label{fig_FOZO}
\end{center}
\vskip -0.2in
\end{figure}

\begin{figure}[t] 
\begin{center}
\centerline{\includegraphics[width=0.9\columnwidth]{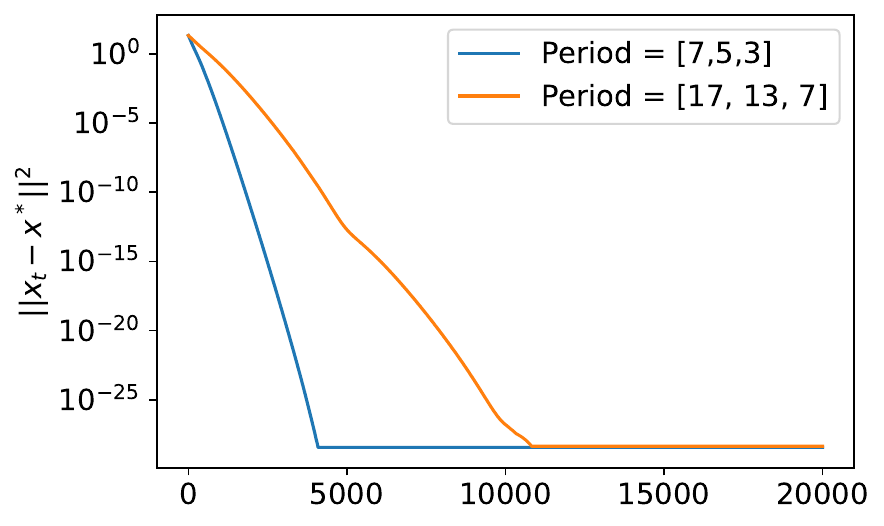}}
\caption{Convergence results of Algorithm~\ref{alg:FO} with update periods [7,5,3] and [17,13,7].}
\label{fig_diffD}
\end{center}
\vskip -0.2in
\end{figure}

\section{Conclusion}\label{sec:6_conclusion}

In this work, we analyzed the discrete-time dynamics of asynchronous games. 
Under the partial asynchronism assumption, we explored the condition that ensures convergence to the Nash equilibrium for partial asynchronous updates.
We proposed the quasidominant condition and observed that asynchronous games may diverge if this condition fails to hold. 
Further, we analyzed the convergence rates for both first- and zeroth-order learning algorithms. 
We provided numerical simulations to illustrate our results. 
Future research directions include improving the convergence rate of the zeroth-order algorithm and exploring some other classes of asynchronous games.

\appendix\label{sec:appendix}
\subsection{Proof of Theorem 1}
Given \eqref{eq:FO:t1}, we have
\begin{align}\label{eq:FO:t3}
    &\left\| x_{i,\tau_{i,t}^m+1} - x_i^* \right\|^2 \nonumber \\
    &\leq \left\| x_{i,\tau_{i,t}^m}-\eta \nabla_i C_i(x_{\tau_{i,t}^m}) - x_i^* \right\|^2 \nonumber \\
    & = \left\| x_{i,\tau_{i,t}^{m-1} +1}-\eta \nabla_i C_i(x_{\tau_{i,t}^m}) - x_i^* \right\|^2 \nonumber \\
    & = \left\| x_{i,\tau_{i,t}^{m-1} +1} - x_i^*\right\|^2  + \eta^2 \left\| \nabla_i C_i(x_{\tau_{i,t}^m}) \right\|^2  \nonumber \\
    & \quad -2\eta\left\langle x_{i,\tau_{i,t}^{m-1} +1} - x_i^*,\nabla_i C_i(x_{\tau_{i,t}^m}) \right\rangle,
\end{align}
where the first inequality follows since the projection operator is nonexpansive and the first equality follows from \eqref{eq:FO:t2}.
Summing up \eqref{eq:FO:t3} over $m=1,\ldots,M_{i,t}$ and rearranging, we have
\begin{align}\label{eq:FO:t4}
    &\Big\| x_{i,\tau_{i,t}^{M_{i,t}}+1} - x_i^* \Big\|^2 \nonumber \\
    &\leq \left\| x_{i,\tau_{i,t}^{0}+1} - x_i^* \right\|^2 + \eta^2 \sum_{m=1}^{M_{i,t}} \left\| \nabla_i C_i(x_{\tau_{i,t}^m}) \right\|^2 \nonumber \\
    &\quad -2\eta  \sum_{m=1}^{M_{i,t}} \left\langle x_{i,\tau_{i,t}^{m-1} +1} - x_i^*,\nabla_i C_i(x_{\tau_{i,t}^m}) \right\rangle.
\end{align}
Since $x_{i,\tau_{i,t}^{M_{i,t}}+1} = x_{i,t+B}$ and $x_{i,\tau_{i,t}^{0}+1} = x_{i,t}$, \eqref{eq:FO:t4} yields
\begin{align}\label{eq:FO:t5}
    &\left\|x_{i,t+B} - x_i^*\right\|^2 \nonumber \\
    &\leq \left\| x_{i,t} - x_i^* \right\|^2 + \eta^2 \sum_{m=1}^{M_{i,t}} \left\| \nabla_i C_i(x_{\tau_{i,t}^m}) \right\|^2 \nonumber \\
    &\quad -2\eta  \sum_{m=1}^{M_{i,t}} \left\langle x_{i,\tau_{i,t}^{m-1} +1} - x_i^*,\nabla_i C_i(x_{\tau_{i,t}^m}) \right\rangle \nonumber \\
    &\leq \left\| x_{i,t} - x_i^* \right\|^2 + \eta^2 M_{i,t} U^2 \nonumber \\
    &\quad -2\eta  \sum_{m=1}^{M_{i,t}} \left\langle x_{i,\tau_{i,t}^{m-1} +1} - x_i^*,\nabla_i C_i(x_{\tau_{i,t}^m}) \right\rangle,
\end{align}
where the last inequality follows from Assumption~\ref{assump:grad:bound}.
For the last term, we have
\begin{align}\label{eq:FO:t6}
    &\left\langle x_{i,\tau_{i,t}^{m-1} +1} - x_i^*,\nabla_i C_i(x_{\tau_{i,t}^m}) \right\rangle \nonumber \\
    & = \left\langle x_{i,\tau_{i,t}^{m-1} +1} - x_i^*,\nabla_i C_i(x_{i,\tau_{i,t}^{m-1}+1},x_{-i,\tau_{i,t}^m}) \right\rangle \nonumber \\
    & = \left\langle x_{i,t} - x_i^* + x_{i,\tau_{i,t}^{m-1} +1} - x_{i,t},\nabla_i C_i(x_{i,\tau_{i,t}^{m-1}+1},x_{-i,\tau_{i,t}^m}) \right\rangle \nonumber \\
    &\geq \left\langle x_{i,t} - x_i^*,\nabla_i C_i(x_{i,\tau_{i,t}^{m-1}+1},x_{-i,\tau_{i,t}^m}) \right\rangle \nonumber \\
    &\quad - \left\| x_{i,\tau_{i,t}^{m-1} +1} - x_{i,t} \right\| \left\| \nabla_i C_i(x_{i,\tau_{i,t}^{m-1}+1},x_{-i,\tau_{i,t}^m}) \right\| \nonumber \\
    &\geq -(m-1) \eta U^2 + \left\langle x_{i,t} - x_i^*,\nabla_i C_i(x_{i,\tau_{i,t}^{m-1}+1},x_{-i,\tau_{i,t}^m}) \right\rangle,
\end{align}
where the first equality follows from \eqref{eq:FO:t2}. 
The last inequality holds since 
\begin{align}\label{eq:FO:t7}
    &\left\|x_{i,\tau_{i,t}^{m-1} +1} - x_{i,t} \right\| =  \left\| x_{i,\tau_{i,t}^{m} } - x_{i,\tau_{i,t}^{1} } \right\| \nonumber \\
    &= \left\|\sum_{k=1}^{m-1} \left( x_{i,\tau_{i,t}^{k+1} } - x_{i,\tau_{i,t}^{k} } \right) \right\| \leq \sum_{k=1}^{m-1}\left\| x_{i,\tau_{i,t}^{k+1} } - x_{i,\tau_{i,t}^{k} } \right\| \nonumber\\
    & = \sum_{k=1}^{m-1}\left\| x_{i,\tau_{i,t}^{k}+1 } - x_{i,\tau_{i,t}^{k} } \right\| \leq \sum_{k=1}^{m-1}  \left\| \eta \nabla_i C_i(x_{\tau_{i,t}^k}) \right\| \nonumber\\
    &\leq \eta (m-1) U.
\end{align}
Besides, we have 
\begin{align}\label{eq:FO:t8}
    &\left\langle x_{i,t} - x_i^*,\nabla_i C_i(x_{i,\tau_{i,t}^{m-1}+1},x_{-i,\tau_{i,t}^m}) \right\rangle \nonumber\\
    & = \Big\langle x_{i,t} - x_i^*, \nabla_i C_i(x_{i,t},x_{-i,\tau_{i,t}^m})  \nonumber\\
    &\quad + \nabla_i C_i( x_{i,\tau_{i,t}^{m-1}+1},x_{-i,\tau_{i,t}^m})-\nabla_i C_i(x_{i,t},x_{-i,\tau_{i,t}^m}) \Big\rangle \nonumber \\
    &\geq \left\langle x_{i,t} - x_i^*,\nabla_i C_i(x_{i,t},x_{-i,\tau_{i,t}^m})\right\rangle  \nonumber \\
    &\quad - DL_{ii}\left\| x_{i,t}- x_{i,\tau_{i,t}^{m-1}+1}\right\| \nonumber \\
    &\geq \left\langle x_{i,t} - x_i^*,\nabla_i C_i(x_{i,t},x_{-i,\tau_{i,t}^m})\right\rangle - DL_{ii} \eta (m-1)U \nonumber \\
    & = \left\langle x_{i,t} - x_i^*,\nabla_i C_i(x_t)\right\rangle - DL_{ii} \eta (m-1)U \nonumber \\
    &\quad + \left\langle x_{i,t} - x_i^*,\nabla_i C_i(x_{i,t},x_{-i,\tau_{i,t}^m}) - \nabla_i C_i(x_t)\right\rangle   \nonumber \\
    &\geq \left\langle x_{i,t} - x_i^*,\nabla_i C_i(x_t)\right\rangle - DL_{ii} \eta (m-1)U \nonumber \\
    &\quad  - D \sum_{j\neq i} L_{ij} \left\| x_{j,t} - x_{j,\tau_{i,t}^m}\right\| \nonumber \\
    &\geq \left\langle x_{i,t} - x_i^*,\nabla_i C_i(x_t)\right\rangle - DL_{ii} \eta (m-1)U \nonumber \\
    &\quad - D\eta B U \sum_{j\neq i} L_{ij},
\end{align}
where the first inequality follows from Assumption~\ref{assump:cost} and the second inequality follows from \eqref{eq:FO:t7}. The last inequality follows since $\left\| x_{j,t} - x_{j,\tau_{i,t}^m}\right\|\leq \sum_{k=t}^{\tau_{i,t}^m -1} \left\| x_{j,k} - x_{j,k+1} \right\| \leq (\tau_{i,t}^m - t)\eta U \leq \eta U B$.

In addition, we have
\begin{align}\label{eq:FO:t9}
    &\left\langle x_{i,t} - x_i^*,\nabla_i C_i(x_t)\right\rangle \nonumber \\
    & \geq  \left\langle x_{i,t} - x_i^*,\nabla_i C_i(x_t) - \nabla_i C_i(x^*)\right\rangle \nonumber \\
    & = \left\langle x_{i,t} - x_i^*,\nabla_i C_i(x_t) - \nabla_i C_i(x_{i}^*,x_{-i,t}) \right\rangle \nonumber \\
    &\quad + \left\langle x_{i,t} - x_i^*,\nabla_i C_i(x_{i}^*,x_{-i,t})- \nabla_i C_i(x^*) \right\rangle \nonumber \\
    &\geq \mu_i \left\| x_{i,t} - x_i^*\right\|^2 \nonumber \\
    &\quad + \left\langle x_{i,t} - x_i^*,\nabla_i C_i(x_{i}^*,x_{-i,t})- \nabla_i C_i(x^*) \right\rangle \nonumber \\
    &\geq \mu_i \left\| x_{i,t} - x_i^*\right\|^2 - \left\| x_{i,t}- x_i^* \right\| \sum_{j\neq i} L_{ij} \left\| x_{j,t} - x_j^*\right\|,
\end{align}
where the first inequality follows from the first-order optimality condition $\langle \nabla_i \mathcal{C}_i(x^{*}), x_i - x_i^{*} \rangle \geq 0$, for all $x_i\in \mathcal{X}_i$, and the remaining inequalities follow from Assumption~\ref{assump:cost}.

Substituting \eqref{eq:FO:t8} and \eqref{eq:FO:t9} into \eqref{eq:FO:t6}, we have
\begin{align}\label{eq:FO:t10}
    &\left\langle x_{i,\tau_{i,t}^{m-1} +1} - x_i^*,\nabla_i C_i(x_{\tau_{i,t}^m}) \right\rangle \nonumber \\
    &\geq -(m-1) \eta U^2 - DL_{ii} \eta (m-1)U - D\eta B U \sum_{j\neq i} L_{ij} \nonumber \\
    &\quad + \mu_i \left\| x_{i,t} - x_i^*\right\|^2 - \sum_{j\neq i} \left\| x_{i,t}- x_i^* \right\|  L_{ij} \left\| x_{j,t} - x_j^*\right\| .
\end{align}
Substituting \eqref{eq:FO:t10} into \eqref{eq:FO:t5} yields
\begin{align}\label{eq:FO:t11}
    &\left\|x_{i,t+B} - x_i^*\right\|^2 \nonumber \\
    &\leq \left\| x_{i,t} - x_i^* \right\|^2 - 2 \eta M_{i,t} \mu_i \left\| x_{i,t} - x_i^* \right\|^2 \nonumber \\
    &\quad + \eta^2 M_{i,t} U^2 +\eta^2 M_{i,t}(M_{i,t}-1) \big( U^2+  L_{ii}DU\big) \nonumber \\
    &\quad + 2 \eta^2 M_{i,t} DBU \sum_{j\neq i} L_{ij}   \nonumber \\
    &\quad + 2 \eta M_{i,t} \sum_{j\neq i} \left\| x_{i,t}- x_i^* \right\|  L_{ij} \left\| x_{j,t} - x_j^*\right\| \nonumber \\
    & =  (1-2\eta M_{i,t} \mu_i) \left\| x_{i,t} - x_i^* \right\|^2 + 2 \eta^2 M_{i,t} DBU \sum_{j\neq i} L_{ij} \nonumber \\
    &\quad  +\eta^2 M_{i,t} U^2+\eta^2 M_{i,t}(M_{i,t}-1) \big( U^2+  L_{ii}DU\big)    \nonumber \\
    &\quad + 2 \eta M_{i,t} \sum_{j\neq i} L_{ij} \left\| x_{i,t}- x_i^* \right\|   \left\| x_{j,t} - x_j^*\right\|.
\end{align}

Recalling the Lyapunov function $V_t = \max_i \frac{\left\| x_{i,t} - x_i^* \right\|^2}{r_i^2}$, it follows that
\begin{align}\label{eq:FO:t12}
    &V_{t+B} = \max_i  \frac{\left\| x_{i,t+B} - x_i^* \right\|^2}{r_i^2} \nonumber \\
    &\leq \max_i \Big\{ (1-2\eta M_{i,t} \mu_i) \frac{\left\| x_{i,t} - x_i^* \right\|^2}{r_i^2} \nonumber \\
    &\quad + 2 \frac{\eta^2}{r_i^2} M_{i,t} DBU \sum_{j\neq i} L_{ij} \nonumber \\
    &\quad + \frac{\eta^2 M_{i,t} U^2}{r_i^2} +\frac{\eta^2}{r_i^2} M_{i,t}(M_{i,t}-1) \big( U^2+  L_{ii}DU\big)    \nonumber \\
    &\quad + 2 \eta M_{i,t} \sum_{j\neq i} L_{ij} \frac{\left\| x_{i,t}- x_i^* \right\|}{r_i}   \frac{\left\| x_{j,t} - x_j^*\right\|}{r_j} \frac{r_j}{r_i} \Big\} \nonumber \\
    &\leq \max_i \Big\{  (1-2\eta M_{i,t} \mu_i) V_t + 2\eta M_{i,t} V_t \sum_{j \neq i} L_{ij}\frac{r_j}{r_i} \nonumber \\
    &\quad + \frac{\eta^2 M_{i,t} U^2}{r_i^2} +\frac{\eta^2}{r_i^2} M_{i,t}(M_{i,t}-1) \big( U^2+  L_{ii}DU\big) \nonumber \\
    &\quad + 2 \frac{\eta^2}{r_i^2} M_{i,t} DBU \sum_{j\neq i} L_{ij} \Big\}  \nonumber \\
    &\leq \max_i \Big\{ \big( 1- 2\eta M_{i,t} \varepsilon ) V_t+ 2 \frac{\eta^2}{r_i^2} M_{i,t} DBU \sum_{j\neq i} L_{ij}  \Big\} \nonumber \\
    &\quad + \frac{\eta^2 M_{i,t} U^2}{r_i^2} +\frac{\eta^2}{r_i^2} M_{i,t}(M_{i,t}-1) \big( U^2+  L_{ii}DU\big)  \nonumber \\
    &\leq \max_i \Big\{ \big( 1- 2\eta \varepsilon ) V_t + \frac{\eta^2}{r_i^2} B U^2 \nonumber \\
    &\quad + \frac{\eta^2}{r_i^2} B^2 \big( U^2+  L_{ii}DU + 2 DU \sum_{j\neq i} L_{ij}\big) \Big\} \nonumber \\
    & \leq (1-2\eta \varepsilon)V_t + \frac{\eta^2}{r_{\min}^2}BU^2 + \eta^2 B^2 c_0,
\end{align}
where the first and second inequalities hold since $\frac{\left\| x_{i,t} - x_i^* \right\|}{r_i} \leq \sqrt{V_t}$, for all $i$. The third inequality follows from the definition $\varepsilon= \min_i \mu_i - \frac{1}{r_i} \sum_{j\neq i}r_j L_{ij}$. The fourth inequality follows since $1\leq M_{i,t}\leq B$ and $\epsilon>0$ due to Assumption 2.
The last inequality follows from the definitions $r_{\min} = \min_i r_i$ and $c_0 = \max_i \frac{U^2+  L_{ii}DU + 2 DU \sum_{j\neq i} L_{ij}}{r_i^2}$.

Without loss of generality, we assume $\frac{T}{B}=H$. Iteratively using \eqref{eq:FO:t12} yields
\begin{align}\label{eq:FO:t13:2}
    V_T &\leq (1-2 \eta \varepsilon)^H V_0 +  \sum_{k=0}^{H-1} (1-2\eta \varepsilon)^k  \eta^2 \big(\frac{BU^2}{r_{\min}^2} +  B^2 c_0  \big)  \nonumber \\
    &\leq (1-2 \eta \varepsilon)^H V_0  + \frac{\eta}{2 \varepsilon}\big(\frac{BU^2}{r_{\min}^2} +  B^2 c_0  \big)  \nonumber \\
    & = (1 - \frac{2\ln H }{H})^H V_0 + \frac{\ln H }{2 \varepsilon^2 H }\big(\frac{BU^2}{r_{\min}^2} +  B^2 c_0  \big) \nonumber \\
    &\leq e^{-2 \ln H} V_0 + \frac{\ln H }{2 \varepsilon^2 H }\big(\frac{BU^2}{r_{\min}^2} +  B^2 c_0  \big) \nonumber \\
    & = \frac{V_0}{H^2} + \frac{\ln H }{2 \varepsilon^2 H }\big(\frac{BU^2}{r_{\min}^2} +  B^2 c_0  \big) \nonumber \\
    & = \mathcal{O}(\frac{ B^2 \ln H}{H}),
\end{align}
where the first equality follows from the definition of $\eta$.
The last inequality follows since $(1-a)^T \leq e^{-a T}$ for any $a \in (0,1)$ and $T> 0$.
Substituting $H=T/B$ into \eqref{eq:FO:t13:2} completes the proof.

% Without loss of generality, we assume $\frac{T}{B}=H$. Summing up \eqref{eq:FO:t12} over $t=1,\ldots,T$ and rearranging, we have
% \begin{align}\label{eq:FO:t13}
%     &\frac{1}{T}\sum_{t=1}^T V_t \leq \sum_{t=1}^T \frac{V_t - V_{t+B}}{2\eta \varepsilon T} + \frac{\eta B U^2} {2 \varepsilon r^2_{\min}} + \frac{\eta B^2 c_0}{2\varepsilon} \nonumber \\
%     &= \sum_{h=0}^{H-1} \sum_{k=1}^{B} \frac{V_{h B+k} - V_{(h+1)B +k}}{2\eta \varepsilon T} + \frac{\eta B U^2} {2 \varepsilon r^2_{\min}} + \frac{\eta B^2 c_0}{2\varepsilon} \nonumber \\
%     &\leq \sum_{k=1}^{B} \frac{V_k}{2\eta \varepsilon T}  + \frac{\eta B U^2} {2 \varepsilon  r^2_{\min}} + \frac{\eta B^2 c_0}{2\varepsilon} \nonumber \\
%     &\leq \frac{B D^2}{2\eta \varepsilon r_{\min}^2 T}  + \frac{\eta B U^2} {2 \varepsilon r^2_{\min}} + \frac{\eta B^2 c_0}{2\varepsilon},
% \end{align}
% where the last inequality follows since $V_t \leq \frac{D^2}{ r_{\min}^2}$.
% Note that $\max_i \left\| x_{i,t} - x_i^*\right\|^2 = \max_i \frac{\left\| x_{i,t} - x_i^*\right\|^2}{r_i^2} r_i^2 \leq V_t \max_i r_i^2$.
% Substituting $\eta = \frac{1}{\sqrt{BT}}$ into \eqref{eq:FO:t13} completes the proof.

\subsection{Proof of Theorem 2}
\begin{proof}
Given that $\nabla_i C_i(x)$ is $L_{ij}$-Lipschitz in $x_j$, we have $\nabla_i C_i(x)$ is Lipschitz continuous in $x$ with the Lipschitz parameter $L_i:= \sum_{j=1}^N L_{ij}$. The statement holds since for any $x,y$, we have $\left\| \nabla_i C_i(x) - \nabla_i C_i(y)\right\| \leq \sum_{j} L_{ij} \left\| x_j - y_j\right\| \leq \sum_{j=1}^n L_{ij} \left\| x - y\right\|$. 

% Then, we have the following property of the function $C_i^{\delta}$, {\color{red} Define} which is presented in the following lemma.
% \begin{lemma}\cite{bravo2018bandit}\label{lemma:cdelta_property:zeroth-order}
% Let Assumption~\ref{assump:cost} hold. Then, we have that
% \begin{enumerate}
%     \item $\mathbb{E}[ \frac{d}{\delta} C_i(\hat{x}_t) u_{i,t}] = \nabla_i C_i^{\delta}(x_t)$;
%     \item $\left\|\nabla_i C_i(x) - \nabla_i C_i^{\delta}(x)\right\| \leq L_i \delta \sqrt{N}$.
% \end{enumerate}
% \end{lemma}

Define $x_{\delta}^{*} := (1-\frac{\delta}{R}) x^{*}$ and $x_{\delta_i}^{*} = (1-\frac{\delta}{R}) x_i^{*}$.
From the update rule \eqref{eq:ZO:t1}, we have
\begin{align}\label{eq:ZO:t2}
    &\left\|x_{i,\tau_{i,t}^m+1} - x_{\delta_i}^*\right\|^2 = \left\| \mathcal{P}_{\mathcal{X}_i^{\delta}}(x_{i,\tau_{i,t}^m}-\eta g_{i,\tau_{i,t}^m} ) - x_{\delta_i}^* \right\|^2 \nonumber \\
    &\leq \left\| x_{i,\tau_{i,t}^m}-\eta g_{i,\tau_{i,t}^m}  - x_{\delta_i}^* \right\|^2 \nonumber \\
    & = \left\| x_{i,\tau_{i,t}^{m-1}+1} - x_{\delta_i}^{*}\right\|^2  + \eta^2 \left\|g_{i,\tau_{i,t}^m} \right\|^2 \nonumber \\
    &\quad - 2 \eta \langle x_{i,\tau_{i,t}^{m-1}+1} - x_{\delta_i}^{*},g_{i,\tau_{i,t}^m}\rangle,
\end{align}
where the second equality follows since $x_{\delta_i}^{*} \in \mathcal{X}_{i}^{\delta}$. Note that $x_{i,\tau_{i,t}^{M_{i,t}}+1} = x_{i,t+B}$ and $x_{i,\tau_{i,t}^{0}+1} = x_{i,t}$.
Summing up \eqref{eq:ZO:t2} over $m=1,\ldots,M_{i,t}$ and rearranging, we have
\begin{align}\label{eq:ZO:t3}
    &\left\|x_{i,t+B} - x_{\delta_i}^*\right\|^2 = \left\|x_{i,\tau_{i,t}^{M_{i,t}}+1} - x_{\delta_i}^*\right\|^2 \nonumber \\
    &\leq \left\| x_{i,\tau_{i,t}^{0}+1} - x_{\delta_i}^* \right\|^2 + \eta^2 \sum_{m=1}^{M_{i,t}} \left\|g_{i,\tau_{i,t}^m} \right\|^2 \nonumber \\
    &\quad -2\eta  \sum_{m=1}^{M_{i,t}} \left\langle x_{i,\tau_{i,t}^{m-1} +1} - x_{\delta_i}^*,g_{i,\tau_{i,t}^m} \right\rangle \nonumber \\
    & = \left\| x_{i,t} - x_{\delta_i}^* \right\|^2 + \eta^2 \sum_{m=1}^{M_{i,t}} \left\|g_{i,\tau_{i,t}^m} \right\|^2  \nonumber \\
    &\quad -2\eta  \sum_{m=1}^{M_{i,t}} \left\langle  x_{i,\tau_{i,t}^{m} } - x_{\delta_i}^*,g_{i,\tau_{i,t}^m} \right\rangle \nonumber \\
    &\leq \left\| x_{i,t} - x_{\delta_i}^* \right\|^2 + \eta^2 M_{i,t} \frac{d^2}{\delta^2} U_c^2 \nonumber \\
    &\quad -2\eta  \sum_{m=1}^{M_{i,t}} \left\langle  x_{i,\tau_{i,t}^{m} } - x_{\delta_i}^*,g_{i,\tau_{i,t}^m} \right\rangle,
\end{align}
where the last inequality follows from Assumption~\ref{assump:C:bound}. 
Recalling that $g_{i,\tau_{i,t}^m}=\frac{d}{\delta} C_i(\hat{x}_{\tau_{i,t}^m}) u_{i,\tau_{i,t}^m}$, we have
\begin{align}\label{eq:ZO:t3:2}
    &\EE_{ u_{i,\tau_{i,t}^m} \sim \SSS} [g_{i,\tau_{i,t}^m} ] \nonumber \\
    &=\frac{d}{\delta} \int_{\SSS} C_i(x_{i,\tau_{i,t}^m }+\delta u_{i,\tau_{i,t}^m }, \hat{x}_{-i,\tau_{i,t}^m}) u_{i,\tau_{i,t}^m } d u_{i,\tau_{i,t}^m } \nonumber\\
    &= \frac{d}{\delta \textrm{vol}_{d-1}(\delta \SSS)} \int_{\delta \SSS} C_i(x_{i,\tau_{i,t}^m }+ u_i', \hat{x}_{-i,\tau_{i,t}^m}) \frac{u_i'}{\left\| u_i'\right\|} d u_i' \nonumber\\
    & = \frac{d}{\delta \textrm{vol}_{d-1}(\delta \SSS)} \nabla_i \int_{\delta \BB} C_i(x_{i,\tau_{i,t}^m }+ w_i, \hat{x}_{-i,\tau_{i,t}^m}) d w_i \nonumber\\
    & = \frac{d}{\delta} \frac{ \textrm{vol}_d(\delta \BB)}{ \textrm{vol}_{d-1}(\delta \SSS)} \nabla_i \int_{ \BB} C_i(x_{i,\tau_{i,t}^m }+ \delta v_i, \hat{x}_{-i,\tau_{i,t}^m}) d v_i \nonumber\\
    & = \EE_{v_i \in \BB} \big[ \nabla_i C_i(x_{i,\tau_{i,t}^m} +\delta v_i, \hat{x}_{-i,\tau_{i,t}^m})  \big],
\end{align}
where the third equality follows from Stoke's theorem, i.e., $\nabla \int_{\delta \BB} f(x + v)  dv = \int_{\delta \SSS} f(x + u) \frac{u}{\|u\|} \, du$, and the last equality follows from the fact that the ratio of volume to surface area of a $d$-dimensional ball of radius $\delta$ is $\delta/d$.

With \eqref{eq:ZO:t3:2}, taking expectation on both sides of \eqref{eq:ZO:t3} with respect to $u_{i,\tau_{i,t}^m}$ yields 
\begin{align}\label{eq:ZO:t4}
    &\EE \left\|x_{i,t+B} - x_{\delta_i}^*\right\|^2 \leq \EE \left\| x_{i,t}  - x_{\delta_i}^* \right\|^2 + \eta^2 M_{i,t} \frac{d^2}{\delta^2} U_c^2 \nonumber \\
    & -2\eta  \sum_{m=1}^{M_{i,t}} \left\langle  x_{i,\tau_{i,t}^{m} } - x_{\delta_i}^*, \EE_{ v_i \in \BB } \big[\nabla_i C_i(x_{i,\tau_{i,t}^m} +\delta v_i, \hat{x}_{-i,\tau_{i,t}^m}) \big]\right\rangle.
\end{align}

For the last term, we have
\begin{align}\label{eq:ZO:t5}
    &\left\langle  x_{i,\tau_{i,t}^{m} } - x_{\delta_i}^*, \EE_{ v_i \in \BB } \big[\nabla_i C_i(x_{i,\tau_{i,t}^m} +\delta v_i, \hat{x}_{-i,\tau_{i,t}^m}) \big]\right\rangle \nonumber \\
    & = \Big\langle  x_{i,\tau_{i,t}^{m} } - x_{\delta_i}^*, \EE_{ v_i \in \BB } \big[\nabla_i C_i(x_{i,\tau_{i,t}^m} +\delta v_i, \hat{x}_{-i,\tau_{i,t}^m}) \big]  \nonumber \\
    &\quad - \nabla_i C_i(x_{\tau_{i,t}^m}) +\nabla_i C_i(x_{\tau_{i,t}^m}) \Big\rangle \nonumber \\
    &\geq  \left\langle  x_{i,\tau_{i,t}^{m} } - x_{\delta_i}^*,\nabla_i C_i(x_{\tau_{i,t}^m})\right\rangle - D L_{i} \delta \sqrt{N} \nonumber \\
    &= \left\langle x_{i,t} - x_{\delta_i}^* + x_{i,\tau_{i,t}^{m} } - x_{i,t}, \nabla_i C_i(x_{\tau_{i,t}^m})\right\rangle - D L_i \delta \sqrt{N} \nonumber \\
    &\geq \left\langle x_{i,t} - x_{\delta_i}^*, \nabla_i C_i(x_{\tau_{i,t}^m})\right\rangle -  (m-1) \eta \frac{d}{\delta} U U_c - D L_i \delta \sqrt{N},
\end{align}
where the first inequality follows since $\nabla_i C_i$ is Lipschitz continuous.
% \begin{align}\label{eq:ZO:t5}
%     &\left\langle  x_{i,\tau_{i,t}^{m} } - x_{\delta_i}^*,\nabla_i C_i^{\delta}(x_{\tau_{i,t}^m}) \right\rangle \nonumber \\
%     & = \left\langle  x_{i,\tau_{i,t}^{m} } - x_{\delta_i}^*,\nabla_i C_i(x_{\tau_{i,t}^m}) +\nabla_i C_i^{\delta}(x_{\tau_{i,t}^m}) -\nabla_i C_i(x_{\tau_{i,t}^m})\right\rangle \nonumber \\
%     &\geq \left\langle  x_{i,\tau_{i,t}^{m} } - x_{\delta_i}^*,\nabla_i C_i(x_{\tau_{i,t}^m})\right\rangle  - D L_i \delta \sqrt{N} \nonumber \\
%     &= \left\langle x_{i,t} - x_{\delta_i}^* + x_{i,\tau_{i,t}^{m} } - x_{i,t}, \nabla_i C_i(x_{\tau_{i,t}^m})\right\rangle - D L_i \delta \sqrt{N} \nonumber \\
%     &\geq \left\langle x_{i,t} - x_{\delta_i}^*, \nabla_i C_i(x_{\tau_{i,t}^m})\right\rangle -  (m-1) \eta \frac{d}{\delta} U U_c - D L_i \delta \sqrt{N},
% \end{align}
The last inequality follows since 
\begin{align*}
    &\left\|x_{i,\tau_{i,t}^{m} } - x_{i,t} \right\| =  \left\| x_{i,\tau_{i,t}^{m} } - x_{i,\tau_{i,t}^{1} } \right\| \nonumber \\
    &= \left\|\sum_{k=1}^{m-1} \left( x_{i,\tau_{i,t}^{k+1} } - x_{i,\tau_{i,t}^{k} } \right) \right\| \leq \sum_{k=1}^{m-1}\left\| x_{i,\tau_{i,t}^{k+1} } - x_{i,\tau_{i,t}^{k} } \right\| \nonumber\\
    & = \sum_{k=1}^{m-1}\left\| x_{i,\tau_{i,t}^{k}+1 } - x_{i,\tau_{i,t}^{k} } \right\| \leq \sum_{k=1}^{m-1}\eta \left\| g_{i,\tau_{i,t}^k} \right\|
     \nonumber\\
    &\leq \sum_{k=1}^{m-1}  \eta \frac{d}{\delta} U_c \leq (m-1) \eta \frac{d}{\delta} U_c.
\end{align*}
Furthermore, we have
\begin{align}\label{eq:ZO:t6}
    &\left\langle x_{i,t} - x_{\delta_i}^*, \nabla_i C_i(x_{\tau_{i,t}^m})\right\rangle \nonumber \\
    & = \left\langle x_{i,t} - x_{\delta_i}^*, \nabla_i C_i(x_{i,t}, x_{-i,\tau_{i,t}^m}) \right\rangle \nonumber \\
    &\quad + \left\langle x_{i,t} - x_{\delta_i}^*,\nabla_i C_i(x_{\tau_{i,t}^m}) - \nabla_i C_i(x_{i,t}, x_{-i,\tau_{i,t}^m}) \right\rangle \nonumber \\
    &\geq \left\langle x_{i,t} - x_{\delta_i}^*, \nabla_i C_i(x_{i,t}, x_{-i,\tau_{i,t}^m}) \right\rangle - L_{ii} D\left\|x_{i,\tau_{i,t}^{m} } - x_{i,t} \right\| \nonumber \\
    &\geq \left\langle x_{i,t} - x_{\delta_i}^*, \nabla_i C_i(x_{i,t}, x_{-i,\tau_{i,t}^m}) \right\rangle - L_{ii} D (m-1) \eta \frac{d}{\delta} U_c \nonumber \\
    & = \left\langle x_{i,t} - x_{\delta_i}^*, \nabla_i C_i(x_t) +\nabla_i C_i(x_{i,t}, x_{-i,\tau_{i,t}^m}) - \nabla_i C_i(x_t)\right\rangle \nonumber \\
    &\quad - L_{ii} D (m-1) \eta \frac{d}{\delta} U_c \nonumber \\
    &\geq \left\langle x_{i,t} - x_{\delta_i}^*, \nabla_i C_i(x_t)\right\rangle - D \sum_{j\neq i} L_{ij} \left\| x_{j, \tau_{i,t}^m} - x_{j,t}\right\| \nonumber \\
    &\quad - L_{ii} D (m-1) \eta \frac{d}{\delta} U_c \nonumber \\
    &\geq \left\langle x_{i,t} - x_{\delta_i}^*, \nabla_i C_i(x_t)\right\rangle - D \sum_{j\neq i} L_{ij} B\eta \frac{d}{\delta} U_c \nonumber \\
    &\quad - L_{ii} D (m-1) \eta \frac{d}{\delta} U_c,
\end{align}
where the last inequality follows since $\left\| x_{j, \tau_{i,t}^m} - x_{j,t}\right\| \leq \sum_{k=t}^{\tau_{i,t}^m -1} \left\| x_{j,k} - x_{j,k+1}\right\| \leq B \eta \frac{d}{\delta}U_c$. Besides, we have
\begin{align}\label{eq:ZO:t7}
    &\left\langle x_{i,t} - x_{\delta_i}^*, \nabla_i C_i(x_t)\right\rangle \nonumber \\
    & = \Big\langle x_{i,t} - x_{\delta_i}^*, \nabla_i C_i(x_t) - \nabla_i C_i (x_{\delta_i}^*,x_{-i,t}) \nonumber \\
    &\quad + \nabla_i C_i (x_{\delta_i}^*,x_{-i,t}) - \nabla_i C_i(x_{\delta}^*) + \nabla_i C_i(x_{\delta}^*) - \nabla_i C_i(x^*) \Big\rangle  \nonumber \\
    &\quad +\left\langle x_{i,t} - x_{\delta_i}^*,\nabla_i C_i(x^*) \right\rangle \nonumber \\
    &\geq \mu_i \left\| x_{i,t} - x_{\delta_i}^*\right\|^2 - \left\| x_{i,t} - x_{\delta_i}^*\right\|  \sum_{j \neq i} L_{ij}\left\| x_{j,t} - x_{\delta_j}^*\right\|  \nonumber \\
    &\quad - D L_i \left\| x_{\delta}^* - x^* \right\| +\left\langle x_{i,t} - x_{\delta_i}^*,\nabla_i C_i(x^*) \right\rangle \nonumber \\
    &\geq \mu_i \left\| x_{i,t} - x_{\delta_i}^*\right\|^2 - \left\| x_{i,t} - x_{\delta_i}^*\right\|  \sum_{j \neq i} L_{ij}\left\| x_{j,t} - x_{\delta_j}^*\right\|  \nonumber \\
    &\quad - D L_i \left\| x_{\delta}^* - x^* \right\| +\left\langle x_{i,t} - x_{\delta_i}^* - (x_{i,t} - x_i^*),\nabla_i C_i(x^*) \right\rangle  \nonumber \\
    & \geq \mu_i \left\| x_{i,t} - x_{\delta_i}^*\right\|^2 - \left\| x_{i,t} - x_{\delta_i}^*\right\|  \sum_{j \neq i} L_{ij}\left\| x_{j,t} - x_{\delta_j}^*\right\|  \nonumber \\
    &\quad - \delta \frac{D^2 L_i}{R} - \delta \frac{D U}{R},
\end{align}
where the second inequality holds due to the first-order optimality condition $\langle x_i - x_i^{*}, \nabla_i \mathcal{C}_i(x^{*}) \rangle \geq 0, \;   {\rm{for \; all}}\;  x_i \in \mathcal{X}_i$. The last inequality follows since $\left\| x_{\delta}^* - x^* \right\| = \frac{\delta}{R}\left\| x^*\right\| \leq \frac{\delta D}{R}$.

Combining \eqref{eq:ZO:t5}, \eqref{eq:ZO:t6} and \eqref{eq:ZO:t7}, we have
\begin{align}\label{eq:ZO:t8}
    &\left\langle  x_{i,\tau_{i,t}^{m} } - x_{\delta_i}^*, \EE_{ v_i \in \BB } \big[\nabla_i C_i(x_{i,\tau_{i,t}^m} +\delta v_i, \hat{x}_{-i,\tau_{i,t}^m}) \big]\right\rangle \nonumber \\
    &\geq \mu_i \left\| x_{i,t} - x_{\delta_i}^*\right\|^2 - \left\| x_{i,t} - x_{\delta_i}^*\right\|  \sum_{j \neq i} L_{ij}\left\| x_{j,t} - x_{\delta_j}^*\right\|  \nonumber \\
    &\quad - \delta \frac{D^2 L_i}{R} - \delta \frac{D U}{R} - D \sum_{j\neq i} L_{ij} B\eta \frac{d}{\delta} U_c - D L_i \delta \sqrt{N} \nonumber \\
    &\quad - L_{ii} D (m-1) \eta \frac{d}{\delta} U_c -  (m-1) \eta \frac{d}{\delta} U U_c .
\end{align}
Substituting \eqref{eq:ZO:t8} into \eqref{eq:ZO:t4}, it yields

\begin{align}\label{eq:ZO:t9}
    &\EE \left\|x_{i,t+B} - x_{\delta_i}^*\right\|^2 \nonumber \\
    &\leq (1 - 2 \eta \mu_i M_{i,t})\EE \left\| x_{i,t}  - x_{\delta_i}^* \right\|^2 + \eta^2 M_{i,t} \frac{d^2}{\delta^2} U_c^2 \nonumber \\
    &\quad + 2 \eta M_{i,t} \left\| x_{i,t} - x_{\delta_i}^*\right\|  \sum_{j \neq i} L_{ij}\left\| x_{j,t} - x_{\delta_j}^*\right\| \nonumber \\
    &\quad + 2\eta \delta M_{i,t} \Big(\frac{D^2 L_i}{R} +\frac{D U}{R}+ D L_i \sqrt{N}\Big) \nonumber \\
    &\quad + \frac{\eta^2 M_{i,t}}{\delta}  \big(2 D \sum_{j\neq i} L_{ij} Bd U_c + (M_{i,t}-1) d U_c (L_{ii} D + U)\big).
\end{align}
Define the Lyapunov functions $V_t^{\delta} = \max_i \frac{\EE \left\|x_{i,t} - x_{\delta_i}^*\right\|^2}{r_i^2}$ and $V_t = \max_i \frac{\EE \left\|x_{i,t} - x_{i}^*\right\|^2}{r_i^2}$. Then, it follows that
\begin{align}\label{eq:ZO:t10}
    &V_{t+B}^{\delta} = \max_i \frac{\EE \left\|x_{i,t+B} - x_{\delta_i}^*\right\|^2}{r_i^2} \nonumber \\
    &\leq \max_i \Big\{  (1 - 2 \eta \mu_i M_{i,t}) \frac{\EE \left\| x_{i,t}  - x_{\delta_i}^* \right\|^2}{r_i^2} + \frac{\eta^2}{\delta^2 r_i^2} M_{i,t} d^2 U_c^2 \nonumber \\
    &\quad + 2 \eta M_{i,t} \frac{\left\| x_{i,t} - x_{\delta_i}^*\right\|}{r_i}  \sum_{j \neq i} L_{ij} \frac{\left\| x_{j,t} - x_{\delta_j}^*\right\|}{r_j} \frac{r_j}{r_i} \nonumber \\
    &\quad + \frac{2\eta \delta M_{i,t}}{r_i^2} \Big(\frac{D^2 L_i}{R} +\frac{D U}{R}+ D L_i \sqrt{N}\Big) \nonumber \\
    &\quad + \frac{\eta^2 M_{i,t} U_c}{\delta r_i^2}  \Big(2 D \sum_{j\neq i} L_{ij} Bd  + (M_{i,t}-1) d (L_{ii} D + U)\Big) \Big\}\nonumber \\
    &\leq \max_i \Big\{ (1 - 2 \eta \mu_i M_{i,t}) V_t^{\delta} + \frac{\eta^2}{\delta^2 } M_{i,t} c_1 \nonumber \\
    &\quad + 2 \eta M_{i,t} V_t^{\delta} \sum_{j \neq i} L_{ij}\frac{r_j}{r_i} + \frac{\eta^2 B^2}{\delta} c_2 + \eta \delta M_{i,t} c_3 \Big\}\nonumber \\
    &\leq \max_i \Big\{ (1 - 2 \eta \varepsilon M_{i,t}) V_t^{\delta}  +\frac{\eta^2 B}{\delta^2} c_1 +\frac{\eta^2 B^2}{\delta} c_2 + \eta \delta B c_3 \Big\}\nonumber \\
    &\leq (1-2\eta \varepsilon) V_t^{\delta} +\frac{\eta^2 B}{\delta^2} c_1 +\frac{\eta^2 B^2}{\delta} c_2 + \eta \delta B c_3,
\end{align}
where we define $c_1 = \max_i \frac{d^2 U_c^2}{r_i^2}$, $c_2 = \max_i \frac{1}{r_i^2}(2 D \sum_{j\neq i} L_{ij} d U_c + d U_c (L_{ii} D + U)) $, and $c_3 = \max_i \frac{2D}{r_i^2}(\frac{D L_i}{R} +\frac{ U}{R}+  L_i \sqrt{N})$. The third inequality follows since $\mu_i -\sum_{j \neq i} L_{ij}\frac{r_j}{r_i} \geq \varepsilon$ and the last inequality follows since $1\leq M_{i,t}\leq B$.

Similarly, we assume $\frac{T}{B}=H$. Iteratively using \eqref{eq:ZO:t10}, we have
\begin{align}
    &V_T^{\delta} \nonumber \\
    &\leq (1- 2 \eta \varepsilon)^H V_0^{\delta}  \nonumber \\
    &\quad + \sum_{k=0}^{H-1} (1-2\eta \varepsilon)^k \big(\frac{\eta^2 B}{\delta^2} c_1 +\frac{\eta^2 B^2}{\delta} c_2 + \eta \delta B c_3 \big) \nonumber \\
    &\leq (1- 2 \eta \varepsilon)^H V_0^{\delta} + \frac{1}{2 \varepsilon} \big(\frac{\eta B}{\delta^2} c_1 +\frac{\eta B^2}{\delta} c_2 + \delta B c_3 \big) \nonumber \\
    & = (1 - \frac{2 \ln H} {H})^H V_0^{\delta} + \frac{\ln H}{H \delta^2} \frac{  B c_1}{2  \varepsilon^2} + \frac{\ln H} {H \delta} \frac{B^2 c_2}{ 2 \varepsilon} +  \frac{\delta B c_3}{2 \varepsilon}
    \nonumber \\
    &\leq \frac{V_0^{\delta}}{H^2} + \frac{\ln H}{H \delta^2} \frac{  B c_1}{2  \varepsilon^2} + \frac{\ln H} {H \delta} \frac{B^2 c_2}{ 2 \varepsilon} +  \frac{\delta B c_3}{2 \varepsilon},
\end{align}
where the last inequality holds since $(1-\frac{2 \ln H}{H})^H \leq e^{-2 \ln H} = H^{-2}$.
Recalling the definition of $V_t$, we have
\begin{align}\label{eq:ZO:final}
    & V_T =  \max_i \frac{\EE \left\|x_{i,T} - x_{i}^*\right\|^2}{r_i^2} \nonumber \\
    & = \max_i \frac{\EE \left\|x_{i,T} - x_{\delta_i}^* + x_{\delta_i}^* - x_{i}^*\right\|^2}{r_i^2} \nonumber \\
    & \leq \max_i \Big\{ \frac{\EE \left\|x_{i,T} - x_{\delta_i}^*\right\|^2 +  2 \EE \left\|x_{i,T} - x_{\delta_i}^*\right\|\left\|x_{\delta_i}^* - x_{i}^*\right\| }{r_i^2} \nonumber \\
    &\quad + \frac{\left\|x_{\delta_i}^* - x_{i}^*\right\|^2}{r_i^2 } \Big\}\nonumber \\
    &\leq \max_i \Big\{ \frac{\EE \left\|x_{i,t} - x_{\delta_i}^*\right\|^2 }{r_i^2} + \frac{2D \delta}{ R r_i^2} + \frac{\delta^2}{R^2 r_i^2}  \Big\}\nonumber \\
    &\leq V_T^{\delta} + \frac{2D \delta}{ R r_{\min}^2} + \frac{\delta^2}{R^2 r_{\min}^2} \nonumber \\
    &\leq \frac{V_0^{\delta}}{H^2} + \frac{B c_1 \ln H}{2  \varepsilon^2 H \delta^2 } + \frac{B^2 c_2 \ln H} {2 \varepsilon H \delta} +  \frac{\delta B c_3}{2 \varepsilon} + \frac{2D \delta}{ R r_{\min}^2} \nonumber \\
    &\quad + \frac{\delta^2}{R^2 r_{\min}^2}.
\end{align}
Substituting $\delta = \frac{B}{T^{1/3}}$ into \eqref{eq:ZO:final} yields the desired result.

\end{proof}

% \section*{Acknowledgment}
% This work is supported in part by the Knut and Alice Wallenberg Foundation, the Swedish Strategic Research Foundation, the Swedish Research Council, AFOSR under award \#FA9550-19-1-0169, and  NSF under award CNS-1932011.

% \ifCLASSOPTIONcaptionsoff
%   \newpage
% \fi



%\bibliographystyle{IEEEtran}


% \end{thebibliography}

\bibliography{00_citation}
\bibliographystyle{unsrt}
% \end{IEEEbiography}


% % insert where needed to balance the two columns on the last page with
% %\newpage



%\vfill


\end{document}